\documentclass{amsart}

\usepackage{amssymb,latexsym}
\usepackage{amsfonts}
\usepackage{multirow}
\usepackage{dsfont}
\usepackage[pdftex]{hyperref}
\usepackage{color}
\usepackage{graphicx}
\usepackage[all]{xy}
\usepackage{ytableau}

\usepackage{tikz}
\usetikzlibrary{calc, shapes, backgrounds,arrows,positioning,plotmarks}
\tikzset{>=stealth',
  head/.style = {fill = white, text=black},
  plaque/.style = {draw, rectangle, minimum size = 10mm, fill=white}, 
     pil/.style={->,thick},
  junct/.style = {draw,circle,inner sep=0.5pt,outer sep=0pt, fill=black}
  }

\theoremstyle{plain}
\newtheorem{theorem}{Theorem}[section]

\newtheorem{conjecture}[theorem]{Conjecture}

\newtheorem{proposition}[theorem]{Proposition}

\newtheorem{examplex}[theorem]{Example}
\newenvironment{example}
  {\pushQED{\qed}\examplex}
  {\popQED\endexamplex}

\theoremstyle{definition}
\newtheorem{remark}[theorem]{Remark}

\numberwithin{equation}{section}

\definecolor{darkblue}{rgb}{0.0,0,0.7}

\newcommand{\newword}[1]{\textcolor{darkblue}{\textbf{\emph{#1}}}}

\newcommand{\BB}{\mathsf{B}}
\newcommand{\PP}{\mathsf{P}}

\newcommand{\bbb}{\mathsf{b}}

\newcommand{\inc}{\ensuremath{\mathrm{Inc}}}

\newcommand{\Frame}{\ensuremath{\mathsf{Frame}}}

\begin{document}

\title{Dynamics of plane partitions:\\ Proof of the Cameron--Fon-Der-Flaass conjecture}

\author{Rebecca Patrias}
\address[RP]{Department of Mathematics, University of St.\ Thomas, St.\ Paul, MN 55105, USA}\email{patr0028@stthomas.edu}

\author{Oliver Pechenik}
\address[OP]{Department of Combinatorics \& Optimization, University of Waterloo, Waterloo, ON N2L 3G1, Canada}\email{oliver.pechenik@uwaterloo.ca}

\date{\today}

\begin{abstract}
One of the oldest outstanding problems in dynamical algebraic combinatorics is the following conjecture of P.~Cameron and D.~Fon-Der-Flaass (1995). Consider a plane partition $P$ in an $a \times b \times c$ box ${\sf B}$. Let $\Psi(P)$ denote the smallest plane partition containing the minimal elements of ${\sf B} - P$. Then if $p= a+b+c-1$ is prime, Cameron and Fon-Der-Flaass conjectured that the cardinality of the $\Psi$-orbit of $P$ is always a multiple of $p$.

This conjecture was established for $p \gg 0$ by Cameron and Fon-Der-Flaass (1995) and for slightly smaller values of $p$ in work of K.~Dilks, J.~Striker, and the second author (2017). 
Our main theorem specializes to prove this conjecture in full generality. 
 \end{abstract}

\maketitle

\section{Introduction}\label{sec:intro}

The relatively young field of \emph{dynamical algebraic combinatorics} studies dynamical properties of actions on various fundamental objects of algebraic combinatorics. For example, alternating sign matrices, plane partitions, root systems, and Young tableaux all carry combinatorially-natural cyclic group actions. In dynamical algebraic combinatorics, one is interested in establishing features of the resulting orbit structures, such as  \emph{cyclic sieving phenomena} \cite{Reiner.Stanton.White}, \emph{homomesies} \cite{Propp.Roby}, \emph{periodicities}, and \emph{resonance phenomena} \cite{Dilks.Pechenik.Striker}. For an excellent survey of the area, see \cite{Striker:notices}.

One of the most studied actions in dynamical algebraic combinatorics is called \emph{rowmotion}. Rowmotion can be defined as an action on the order ideals of any finite poset $\PP$. Interesting dynamical properties appear when $\PP$ is chosen to be a poset of significance in algebraic combinatorics. While much of the dynamical algebraic combinatorics literature dates from the past 15 or so years, rowmotion has older roots; it first appeared in 1974 through independent work of P.~Duchet \cite{Duchet} (in a special case) and of A.~Brouwer and A.~Schrijver \cite{Brouwer.Schrijver} (in full generality).

One of the oldest open problems in dynamical algebraic combinatorics has been a 1995 conjecture of P.~Cameron and D.~Fon-Der-Flaass \cite{Cameron.Fonderflaass} on the periodicity of rowmotion for plane partitions. The main goal of this paper is to prove their conjecture, which we now recall.

Fix positive integers $a,b,c \in \mathbb{Z}^+$ and consider plane partitions sitting inside a rectangular $a \times b \times c$ box. We identify this box with the poset $\BB_{a,b,c} = {\bf a} \times {\bf b} \times {\bf c}$ that is the product of three chains, and identify plane partitions in this box with order ideals of the poset $\BB_{a,b,c}$. 

We write $J(\PP)$ for the set of all order ideals of a poset $\PP$. Given $I \in J(\PP)$, define $\Psi(I)$ to be the order ideal generated by the minimal elements of the complementary order filter $\PP - I$. Following \cite{Striker.Williams}, we refer to the operator $\Psi$ as \newword{rowmotion}. It is straightforward to see that the action of $\Psi$ is reversible, so it permutes the elements of $J(\PP)$ and partitions them into disjoint orbits. For a general poset $\PP$, these orbits tend to be large and without discernible structure. However, for special posets $\PP$, intricate structure has been discovered; see, e.g., \cite{Armstrong.Stump.Thomas,Brouwer.Schrijver,Cameron.Fonderflaass,Mandel.Pechenik,Panyushev,Propp.Roby,Rush.Shi,Striker.Williams,Vorland} for various such results.

Cameron and Fon-Der-Flaass \cite{Cameron.Fonderflaass} made the following periodicity conjecture for rowmotion on the poset $\BB_{a,b,c}$.

\begin{conjecture}[{\cite{Cameron.Fonderflaass}}]\label{conj:main}
	Suppose $p=a+b+c-1$ is prime. Then the cardinality of every $\Psi$-orbit of $J(\BB_{a,b,c})$ is a multiple of $p$.
\end{conjecture}

\begin{remark}
	Conjecture~\ref{conj:main} proposes a special kind of \emph{resonance} in the sense of \cite{Dilks.Pechenik.Striker}. That is, while the $\Psi$-orbit cardinalities remain unknown, they all ``resonate with the frequency $p$,'' being all of the form $hp$ for some positive integers $h$. It would be very interesting to understand the values $h$ that appear. Experimentally, there appears to be a strong bias toward odd values of $h$. We currently have no explanation for this phenomenon, nor do we have good upper bounds on the values $h$.
\end{remark}

Our main result is the following, which implies Conjecture~\ref{conj:main}. 

\begin{theorem}\label{thm:main}
	Let $k$ be the cardinality of any $\Psi$-orbit of $J(\BB_{a,b,c})$. Then 
	\[
	\gcd(k,a+b+c-1) > 1.
	\]
\end{theorem}

Previous work had succeeded in establishing Conjecture~\ref{conj:main} only for very small and very large values of $c$.
The case $c=1$ was established earlier by Brouwer and Schrijver \cite{Brouwer.Schrijver} and the case $c=2$ by Cameron and Fon-Der-Flaass \cite{Cameron.Fonderflaass}. (Indeed, in these ``small $c$'' cases the size of every $\Psi$-orbit is exactly $p$.) Cameron and Fon-Der-Flaass \cite{Cameron.Fonderflaass} also established the ``large $c$'' case 
$
c > ab - a - b + 1.
$
This bound was later improved to 
\[
c > \frac{2ab - 2}{3} -a - b + 2
\] 
in \cite[Theorem~4.13]{Dilks.Pechenik.Striker}.

Our superficially short proof of Theorem~\ref{thm:main} and Conjecture~\ref{conj:main} is uniform and does not rely on any of these previous partial results. Nonetheless, we are heavily indebted to previous work that was not available when Cameron and Fon-Der-Flaass first made their conjecture. Explicitly, our proof calls upon some of the main results of \cite{Dilks.Pechenik.Striker} and \cite{Pechenik:frames}. In a deeper sense, our proof builds on technology and theorems developed previously in the various other papers \cite{Pechenik:CSP, Striker.Williams, Thomas.Yong:K, Thomas.Yong:Plancherel}, as well. 

More specifically, in Section~\ref{sec:reformulation}, we  use the results of \cite{Dilks.Pechenik.Striker} to translate Theorem~\ref{thm:main} into an equivalent statement about the combinatorics of \emph{$K$-promotion} on \emph{increasing tableaux}. $K$-promotion was first studied in \cite{Pechenik:CSP} as an outgrowth of the combinatorics of $K$-theoretic Schubert calculus for Grassmannians introduced in \cite{Thomas.Yong:K}, and has since been studied in several purely combinatorial contexts. In Section~\ref{sec:proof}, we then prove this translated conjecture, relying on the main theorem of \cite{Pechenik:frames}.

\section{Reformulation in terms of increasing tableaux}\label{sec:reformulation}

Our first step in proving Theorem~\ref{thm:main} is to use the results of \cite{Dilks.Pechenik.Striker} to translate it into an equivalent statement regarding different combinatorics. First, we recall the definitions of increasing tableaux and the $K$-promotion operator on them.

We write $a \times b$ to denote the grid of boxes with $a$ rows and $b$ columns.  Equivalently, this is the Young diagram of the partition with $a$ parts all of size $b$.
Index the boxes of $a \times b$ as in a matrix, so the box $(1,2)$ is the box in the second column from the left in the top row. For a box $\bbb$ in $a \times b$, we write $\bbb^\rightarrow$ for the box immediately right of $\bbb$, $\bbb^\downarrow$ for the box immediately below $\bbb$, etc. A \newword{short ribbon} in $a \times b$ is an edge-connected subset of boxes with at most two in any row or column.

An \newword{increasing tableau} of shape $a \times b$ is a filling $T$ of the boxes of $a \times b$ with positive integers, so that rows strictly increase from left to right and columns strictly increase from top to bottom. That is, for every box $\bbb$, we have $T(\bbb) < T(\bbb^\rightarrow)$ and $T(\bbb) < T(\bbb^\downarrow)$. 
We write $\inc(a \times b)$ for the set of all increasing tableaux of shape $a \times b$ and write $\inc^q(a \times b)$ for the finite subset with entries at most $q$. Note that in an increasing tableau, if we look at the set of boxes containing either $i$ or $i+1$, the edge-connected components of this set are all short ribbons.

\begin{example}
An increasing tableau of shape $3 \times 6$ is
\[
T = \ytableaushort{12356{10}, 24589{11},679{10}{13}{17}} \in \inc^{17}(3 \times 6).
\]
Note that not every number from $1$ to $17$ need appear. Note also that, for example, the boxes labeled $4$ and $5$ make up two short ribbons, while the boxes labeled $1$ and $2$ make up a single short ribbon.
\end{example}

Following \cite{Buch.Samuel}, we say that $T \in \inc(a \times b)$ is \newword{minimal} if we have 
\begin{itemize}
	\item $T(1,1) = 1$,
	\item $T(\bbb^\rightarrow) = T(\bbb) + 1$ (for all $\bbb$ not in the rightmost column), and
	\item $T(\bbb^\downarrow) = T(\bbb) + 1$ (for all $\bbb$ not in the bottom row).
\end{itemize}
Note that there is a unique minimal tableau $M_{a \times b}$ of each shape $a \times b$, and that $M_{a \times b}$ is the unique element of $\inc^{a+b-1}(a \times b)$. Moreover, $\inc^q(a \times b)$ is empty if $q < a+b-1$. 

We now recall the definition of \newword{$K$-promotion} on increasing tableaux.  Let $T \in \inc^q(a \times b)$. Consider the short ribbons consisting of the boxes labeled $1$ and $2$. Say a short ribbon is \newword{trivial} if it consists of a single box, and \newword{nontrivial} otherwise. For each trivial short ribbon, we do nothing, while for each nontrivial short ribbon, we swap the labels $1$ and $2$. The result is generally not an increasing tableau, but nonetheless consider the short ribbons in it consisting of the boxes labeled $1$ and $3$ and repeat this process, successively swapping the pairs of labels $(1,4),(1,5), \dots, (1,q)$ in nontrivial short ribbons. Note that, if the box in position $(1,1)$ originally had label $1$, then label $1$ finally appears only in position $(a,b)$. To finish, decrement the label in each box by $1$, and replace any resulting $0$ label by $q$. The result is now an increasing tableau in $\inc^q(a \times b)$, the $K$-promotion of $T$. See Example~\ref{ex:promotion} for an example of this process. We will abuse notation by also denoting the $K$-promotion of $T$ by $\Psi(T)$, as there can be no confusion with rowmotion of plane partitions. We write $\Psi^\bullet(T)$ to denote the $\Psi$-orbit of the increasing tableau $T$.

\begin{remark}
	Increasing tableaux are a special case of the more classically studied \emph{semistandard tableaux} and $K$-promotion shares features with M.-P.~Sch\"utzenberger's \emph{promotion} (see \cite{Schutzenberger}) for semistandard tableaux; however, promotion of semistandard tableaux does not preserve the subset of increasing tableaux and $K$-promotion does not coincide with promotion.
\end{remark}

\begin{example}\label{ex:promotion} Starting with the tableau $T\in \inc^9(4\times 4)$ shown below, we illustrate the process of computing its $K$-promotion $\Psi(T)$. At each step, trivial short ribbons are shown in light gray and nontrivial short ribbons are shown in darker gray. 
\ytableausetup{boxsize=1.4em}
\begin{align*}T&=\begin{ytableau}*(gray) 1 & *(gray) 2 & 4 & 5 \\ *(gray) 2 & 3 & 5 & 6\\ 4 & 5 & 7 & 8\\ 5 & 6 & 8 & 9\end{ytableau}\rightarrow\begin{ytableau}2 & *(gray) 1 & 4 & 5 \\ *(gray) 1 & *(gray) 3 & 5 & 6\\ 4 & 5 & 7 & 8\\ 5 & 6 & 8 & 9\end{ytableau}\rightarrow\begin{ytableau}2 & 3 & *(lightgray) 4 & 5 \\ 3 & *(lightgray) 1 & 5 & 6\\ *(lightgray) 4 & 5 & 7 & 8\\ 5 & 6 & 8 & 9\end{ytableau}\rightarrow\begin{ytableau}2 & 3 &  4 & *(lightgray) 5 \\ 3 & *(gray) 1 & *(gray) 5 & 6\\ 4 & *(gray) 5 & 7 & 8\\ *(lightgray) 5 & 6 & 8 & 9\end{ytableau}\rightarrow\begin{ytableau}2 & 3 &  4 & 5 \\ 3 & 5 & *(gray) 1 & *(gray) 6\\ 4 & *(gray) 1 & 7 & 8\\ 5 & *(gray) 6 & 8 & 9\end{ytableau} \\ &\rightarrow \begin{ytableau}2 & 3 &  4 & 5 \\ 3 & 5 & 6 & *(lightgray) 1\\ 4 & 6 & *(lightgray) 7 & 8\\ 5 & *(lightgray) 1 & 8 & 9\end{ytableau} \rightarrow \begin{ytableau}2 & 3 &  4 & 5 \\ 3 & 5 & 6 & *(gray) 1\\ 4 & 6 & 7 & *(gray) 8\\ 5 & *(gray) 1 & *(gray) 8 & 9\end{ytableau} \rightarrow  \begin{ytableau}2 & 3 &  4 & 5 \\ 3 & 5 & 6 & 8 \\ 4 & 6 & 7 & *(gray) 1\\ 5 & 8 & *(gray) 1 & *(gray) 9\end{ytableau} \rightarrow 
 \begin{ytableau}2 & 3 &  4 & 5 \\ 3 & 5 & 6 & 8 \\ 4 & 6 & 7 & 9\\ 5 & 8 & 9 & 1\end{ytableau} \rightarrow 
  \begin{ytableau}1 & 2 &  3 & 4 \\ 2 & 4 & 5 & 7 \\ 3 & 5 & 6 & 8\\ 4 & 7 & 8 & 9\end{ytableau}=\Psi(T)\end{align*}
\end{example}
\ytableausetup{boxsize=normal}

By \cite[Theorem~4.4]{Dilks.Pechenik.Striker}, there is a $\Psi$-equivariant bijection between the sets $J(\BB_{a,b,c})$ and $\inc^{a+b+c-1}(a \times b)$. Hence, to prove Theorem~\ref{thm:main} and Conjecture~\ref{conj:main}, it is sufficient to establish the following.

\begin{theorem}\label{thm:tab}
Let $q > a+b -1$ and suppose the $\Psi$-orbit of $T \in \inc^q(a \times b)$ has cardinality $k$. Then $
\gcd(k, q) > 1.
$
\end{theorem}

\begin{remark}
	The hypothesis $q > a+b-1$ in Theorem~\ref{thm:tab} is necessary merely to exclude the minimal tableau $M_{a \times b}$, corresponding under the $\Psi$-equivariant bijection of \cite{Dilks.Pechenik.Striker} to the empty plane partition in the degenerate $a \times b \times 0$ box. Obviously, these objects have $\Psi$-orbits of size $1$. 
\end{remark}

\section{Proof of Theorem~\ref{thm:tab}}\label{sec:proof}

Let $q \geq a+ b - 1$ and fix $T \in \inc^q(a \times b)$. Suppose $|\Psi^\bullet(T)| = k$ and $\gcd(k,q) = 1$. We aim to show that $T$ is minimal, so $q = a+b-1$.

The \newword{frame} of the shape $a \times b$ is the set $\Frame(a \times b)$ consisting of those boxes in the first or last column, or first or last row, of $a \times b$. The \newword{frame} $\Frame(U)$ of the tableau $U \in \inc^q(a \times b)$ is the restriction of the filling $U$ to $\Frame(a \times b)$.

\begin{example}
	For $T$ as in Example~\ref{ex:promotion}, the frame consists of the boxes shaded in light gray below.
	\[\begin{ytableau}*(lightgray) 1 & *(lightgray) 2 & *(lightgray) 4 & *(lightgray) 5 \\ *(lightgray) 2 &  3 &   5 &  *(lightgray) 6\\ *(lightgray) 4 &   5 & 7 &  *(lightgray) 8\\ *(lightgray) 5 &  *(lightgray) 6 & *(lightgray)  8 &  *(lightgray) 9\end{ytableau}\]
\end{example}

Consider the cyclic group $C_k = \langle \psi \rangle$ of order $k$. Define an action of $C_k$ on $\Psi^\bullet(T)$ by $\psi \cdot U = \Psi(U)$ for all $U \in \Psi^\bullet(T)$.
 Since $k$ and $q$ are relatively prime, the group element $\psi^q$ also generates $C_k$. Hence, every $U \in \Psi^\bullet(T)$ is of the form $\Psi^{mq}(T)$ for some positive integer $m \in \mathbb{Z}^+$. 
 
 By \cite[Theorem~2]{Pechenik:frames}, we have $\Frame(U) = \Frame(\Psi^q(U))$ for all tableaux $U \in \inc^q(a \times b)$. 
 Hence, by the observation of the previous paragraph, we have $\Frame(U) = \Frame(T)$ for every $U \in \Psi^\bullet(T)$. In particular, $\Frame(T) = \Frame(\Psi(T))$.
 
 The condition $\Frame(T) = \Frame(\Psi(T))$ turns out to be very strict. Indeed, the following proposition  implies that $T$ is therefore a minimal tableau and $q = a+b-1$, completing the proof of Theorem~\ref{thm:tab}.

\begin{proposition}\label{prop:sameframe}
	Suppose $V \in \inc^\ell(a \times b)$ satisfies $\Frame(V) = \Frame(\Psi(V))$. Then $V$ is minimal and $\ell = a+b-1$.
\end{proposition}

Before proving this proposition, we need a few more definitions. Let $V\in \inc^\ell(a×\times b)$. Following \cite{Dilks.Pechenik.Striker}, we define the \newword{flow path} of $V$ to be the set of pairs $\{\bbb,\bbb'\}$ of adjacent boxes of $V$ such that $\bbb$ and $\bbb'$ are at some point part of the same nontrivial short ribbon during the application of $K$-promotion to $V$. We define the \newword{stream-bed} of $V$ to be the union of the flow path, i.e., the set of all boxes $\bbb$ appearing in any pair $\{\bbb, \bbb' \}$ of the flow path of $V$. (Warning: In \cite{Pechenik:CSP}, the term ``flow path'' was used to refer to what we here call a ``stream-bed.'') Observe that if $\bbb \neq (1,1)$ is in the stream-bed of $V$, then either $\{\bbb^\leftarrow, \bbb \}$ or $\{ \bbb^\uparrow, \bbb \}$ is in the flow path of $V$. Similarly observe that if $\bbb \neq (a, b)$ is in the stream-bed of $V$, then either $\{ \bbb, \bbb^\rightarrow \}$ or $\{ \bbb, \bbb^\downarrow \}$ is in the  flow path of $V$.

\begin{example}\label{ex:flow}
Let $T$ be as in Example~\ref{ex:promotion}. Then its stream-bed is the union of all the darker gray short ribbons in all the tableaux illustrated there.
\[\begin{ytableau}*(gray) 1 & *(gray) 2 & 4 & 5 \\ *(gray) 2 & *(gray) 3 &  *(gray) 5 &  *(gray) 6\\ 4 &  *(gray) 5 & 7 &  *(gray) 8\\ 5 &  *(gray) 6 & *(gray)  8 &  *(gray) 9\end{ytableau}\]
\end{example}

\begin{proof}[Proof of Proposition~\ref{prop:sameframe}]
We have $V(1,1) = 1$, for otherwise $\Psi(V)(1,1) = V(1,1) - 1$, contradicting $\Frame(V) = \Frame(\Psi(V))$. For any tableau $W \in \inc^\ell(a \times b)$ with $W(1,1) = 1$, we have $\Psi(W)(a,b) = \ell$. Hence, by $\Frame(V) = \Frame(\Psi(V))$, we also have $V(a,b) = \ell$.

	Consider $\bbb \in \Frame(a \times b)$. If $\bbb$ is not in the stream-bed of $V$, then $\Psi(V)(\bbb) = V(\bbb) - 1$, contradicting $\Frame(V) = \Frame(\Psi(V))$. Hence, every box of $\Frame(a \times b)$ must be in the stream-bed of $V$.
	
	Consider $\{\bbb, \bbb^\rightarrow\}$ in the top row of $V$. Since $\bbb^\rightarrow$ is in the stream-bed of $V$, the pair $\{\bbb, \bbb^\rightarrow\}$ must be in the flow path of $V$. Hence $\Psi(V)(\bbb) = V(\bbb^\rightarrow) - 1$. But by assumption $\Psi(V)(\bbb) = V(\bbb)$, so we have $V(\bbb^\rightarrow) = V(\bbb) + 1$. Similarly, we have $V(\bbb^\downarrow) = V(\bbb) + 1$ for $\bbb$ in the leftmost column of $V$.
	
	Consider $\{\bbb, \bbb^\rightarrow\}$ in the bottom row of $V$. Since $\bbb$ is in the stream-bed of $V$, the pair $\{\bbb, \bbb^\rightarrow\}$ must be in the flow path of $V$. Thus again we have $V(\bbb^\rightarrow) = V(\bbb) + 1$. Similarly, we have $V(\bbb^\downarrow) = V(\bbb) + 1$ for $\bbb$ in the rightmost column of $V$.

Therefore, the entries of $V$ increase consecutively around $\Frame(a \times b)$ from upper-left to lower-right. In particular, the largest entry of $V$ must be $a + b - 1$. But we already determined that this largest entry was $\ell$ in position $(a,b)$. Hence, $\ell = a + b - 1$ and $V$ is the minimal tableau $M_{a \times b}$.
\end{proof}

\section*{Acknowledgements}
OP was partially supported by a Mathematical Sciences Postdoctoral Research Fellowship (\#1703696) from the National Science Foundation.

The authors are grateful to Hugh Thomas for comments on an earlier draft of this paper.

\bibliographystyle{amsalpha} 
\bibliography{CFDF}

\end{document}